\documentclass[11pt, oneside]{article}   	

\usepackage{amsmath,amssymb,latexsym,amsthm,enumerate,color}

\numberwithin{equation}{section}

\newtheorem{theorem}{Theorem}[section]
\newtheorem{lemma}[theorem]{Lemma}

\newtheorem{problem}{Problem}

\theoremstyle{definition}
\newtheorem{definition}[theorem]{Definition}

\usepackage{geometry}  
\usepackage{ amsthm}            		
\usepackage{graphicx}				
\usepackage{ams math}								
\usepackage{amssymb}
\usepackage{amsthm}
\usepackage{setspace}
\usepackage{abstract}

\usepackage{cite}
\onehalfspace

\usepackage{cite}
\usepackage{amsthm}

\def\Sym{{\rm Sym}}
\def\PSL{{\rm PSL}}
\def\AGL{{\rm AGL}}

\begin{document}

\title{Circulant association schemes on triples}
\author{Prabir Bhattacharya\\ 
School of Arts, Science \& Technology, Thomas Edison State University,\\ 
Trenton, New Jersey, USA\\ and \\
Cheryl E Praeger\\
Department of Mathematics and Statistics, University of Western Australia,\\ 
Perth, Australia\\ \\
Dedicated to the memory of Vaughan Jones}

\maketitle




\begin{abstract}
Association Schemes and coherent configurations (and the related Bose-Mesner algebra  and  coherent algebras) are well known in combinatorics with many applications.
 In the 1990s,  Mesner and Bhattacharya introduced a three-dimensional generalisation of association schemes which they called an {\em association scheme on triples} (AST) and constructed examples of several families of ASTs. Many of their examples used 2-transitive permutation groups:  the non-trivial ternary relations of the ASTs were sets of ordered triples of pairwise distinct points of the underlying set left invariant by the group; and the given permutation group was a subgroup of automorphisms of the AST.  In this paper, we consider ASTs that do not necessarily admit 2-transitive groups as automorphism groups but instead a transitive cyclic subgroup of the symmetric group acts as automorphisms. Such ASTs are called {\em circulant} ASTs and the corresponding ternary relations are called {\em circulant relations}. We give a complete characterisation of circulant ASTs in terms of AST-regular partitions of the underlying set. We also show that a special type of circulant,  that we call a {\em thin circulant}, plays a key role in describing the structure of circulant ASTs. We outline several open questions. 
 
Funding\footnote{Funding Acknowledgement: Australian Research Council Discovery Project DP200100080.}.
\end{abstract}

\section{Introduction}

\subsection{Tribute to V.~F.~R.~Jones by the second author} I had the pleasure of meeting Vaughan Jones, to whose memory this paper is dedicated, on several occasions. We served together on a chair selection committee for the University of Auckland, we met at several summer schools of the New Zealand Mathematical Research Institute (NZMRI), and at international meetings such as the 2014 International Congress of Mathematicians when Vaughan became an IMU Vice President. I also recall the splendid public lecture Vaughan gave at a joint meeting of the Australian Mathematical Society and the New Zealand Mathematical Society.  The overwhelming impression Vaughan left on me, and no doubt on others, was his \emph{joie de vivre}, his enthusiastic embrace of life. I remember Vaughan as the sunburnt surfer at  NZMRI summer schools which were invariably held at locations where Vaughan could enjoy surfing daily.  Vaughan's abiding commitment to mathematics in New Zealand was evident. He is sorely missed.

Vaughan Jones' mathematical interests were very broad. Although we worked in different areas, this paper reports on new developments which we hope Vaughan would have enjoyed, namely a (slightly) higher dimensional version of the association schemes or Bose--Mesner algebras which he mentioned in his substantial 1999 preprint \cite[pp. 48-49]{VJ}   as a way of constructing planar algebras. This work \cite{VJ} of Vaughan's  is readily available on the arXiv and has received $455$ Google Scholar citations (up to 31 March 2021). However, although listed in various papers, such as \cite{Gh}, as being `New Zealand J. Math., in press',  it remains unpublished.  

\subsection{Association schemes on triples}
Association schemes and coherent configurations (\hspace{-0.1cm} \cite{ABC, Bose52, Higman, Higman2}), and the related Bose-Mesner algebras and coherent algebras (\hspace{-0.1cm} \cite{B-M})
are well-known  in combinatorics with 
many applications in coding theory, graph theory and statistics (e.g., \cite{Bailey, 
Bannai18, Bannai84, Bannai93,
Cameron, Delsarte73, Delsarte, Godsil2010, Godsil2016,
Camion,  MacW, Klin, Seidel, Zie1}).

In 1990, Dale Mesner and the first author introduced a $3$-dimensional analogue of associations schemes, which they called \emph{association schemes on triples} (in \hspace{-.02cm}\cite{MB90, MB94}). 

\begin{definition}  {\rm (\hspace{-0.1cm} \cite{MB90})}.
\label{AST1}
{\rm An {\it association scheme on triples} (in short, an AST) is a pair $(\mathcal{A}, \Omega)$, where 
$\Omega$ is a set of cardinality $n\geq3$, and $\mathcal{A}$ is a partition of 
$\Omega^3=\Omega \times \Omega \times \Omega$ into (ternary) relations $R_0, \ldots, R_m$  ($m \geq4$),  such that 
\begin{align*}
R_0 &= \{ (x,x,x): x \in \Omega \}, \\
R_1 &= \{(x,y, y): x, y \in \Omega, x \neq y \}, \\
R_2 &= \{(x,y, x): x, y \in \Omega, x \neq y \},\\
R_3 &= \{(x,x, y): x, y \in \Omega, x \neq y \},
\end{align*}

and the following three conditions hold.
\begin{enumerate} 

 \item[A1] For each $i\in \{4, \cdots, m \}$, there is a positive constant $n_i^{(3)}$ such that, for each pair $x, y\in\Omega$ with $x \neq y$, there are exactly $n_i^{(3)}$ elements $z\in \Omega$ for which $(x,y,z) \in R_i$. 

\item[A2] (Principal regularity condition) For any (not necessarily distinct) elements $i, j, k, \ell\in \{0, \cdots, m \}$, there is a constant $p_{ijk}^\ell$ such that, for each $(x,y,z) \in R_\ell$, there are exactly $p_{ijk}^\ell$ elements $w \in \Omega$ such that $(w,y,z) \in R_i, (x,w,z) \in R_j,$ and $(x,y,w) \in R_k$.

\item[A3] For each $i \in \{0, \cdots, m \}$ and each permutation $\sigma$ of $\{1,2,3\}$, there exists 
$j \in \{0, \cdots, m \}$ such that
$$
R_j=\{(x_{1 \sigma}, x_{2 \sigma}, x_{3 \sigma}) : (x_1, x_2, x_3) \in R_i \}.
$$

\end{enumerate}
The ternary relations $R_0, R_1, R_2$ and $R_3$ are called the {\em trivial relations}; and the other ternary relations of the AST are called the {\em nontrivial} relations.
}
\end{definition}

The definition provides a minimal set of conditions on an AST, and these imply various other restrictions. For example,  
conditions A1 and A3 jointly imply the existence of constants $n_i^{(1)}, n_i^{(2)}$ such that, for each pair $x, y\in\Omega$ with $x \neq y$, there are exactly $n_i^{(1)}$ elements $z\in \Omega$ for which $(z,x,y) \in R_i$, and exactly $n_i^{(2)}$ elements $z\in \Omega$ for which $(x,z,y) \in R_i$.  There are also relations between these constants  and the $p_{ijk}^\ell$. For example, it was observed in \cite[Lemma 2.1]{MB90} that, for $4\leq i\leq m$,
\begin{equation}\label{E:MB1}
n_i^{(1)} =\sum_{k=0}^m p_{i2k}^2,\quad n_i^{(2)} = \sum_{k=0}^m p_{1ik}^1.
\end{equation} 
Condition A3 yields an action of $\Sym(3)$ on $\mathcal{A}$ given by $\sigma:R_i\to R_{i}^\sigma$, where $R_{i}^\sigma$ is the relation $R_j$ in condition A3. If $R_{i}^\sigma=R_i$ for each $\sigma\in\Sym(3)$, then $R_i$ is called a \emph{symmetric} relation, and moreover, if all the $R_i$ 
($i\geq4$) are symmetric then the AST is called symmetric. 
Symmetric ASTs were studied in \cite{MB90}, where further restrictions were obtained on the structure constants $p_{ijk}^\ell$, and links were explored between symmetric ASTs and families of $2$-designs with block-size $3$.

Several infinite families of ASTs were constructed in \cite{MB90}. Notably it was shown in   \cite[Theorem 4.1]{MB90} that, for each $2$-transitive permutation group $G$ on $\Omega$, we obtain as AST by taking as the nontrivial ternary relations the $G$-orbits on the set of   triples of pairwise distinct points of $\Omega$. The group $G$ is then admitted as a group of automorphisms of the AST. The structure of a number of specific examples constructed in this way was explored in detail in \cite[Section 4]{MB90}.

\subsection{Circulant ternary relations and circulant ASTs}\label{sub:circ}

In this paper we seek further examples which may not admit an automorphism group which is $2$-transitive on $\Omega$. Instead we ask that they admit a transitive cyclic subgroup of $\Sym(\Omega)$, and without loss of generality we assume that $\Omega=\{0,\dots,n-1\}$ and that the transitive cyclic group $A=\langle a\rangle$ is generated by the map $a:x\to x+1$ (modulo $n$), which acts as an $n$-cycle on $\Omega$. 
Also we ask that each relation $R$ in these ASTs should be invariant under the action of $A$, that is $R$ should satisfy:

\begin{equation}\label{eq:3circ}
\mbox{$(i,j,k)\in R$ implies that $(i+1, j+1, k+1)\in R$.}
\end{equation}
Clearly this condition holds for each of the trivial relations $R_0, R_1, R_2, R_3$ in Definition~\ref{AST1}; it is the nontrivial ternary relations in an AST which conern us.

\begin{definition}\label{d:3circ}
{\rm  
A ternary relation $R\subseteq \Omega^3$ satisfying \eqref{eq:3circ} is called a \emph{circulant $3$-relation} or simply a \emph{$3$-circulant}. If each of the nontrivial relations of an AST $\mathcal{A}$ is a $3$-circulant,   then we say that  $\mathcal{A}$ is a \emph{circulant AST}.
}
\end{definition}

Indeed, some of the  explicit examples of ASTs constructed in \cite{MB90} from $2$-transitive permutation groups are circulants. For example, the AST in \cite[Theorem 4.4]{MB90} constructed from the $2$-dimensional projective group $\PSL(2,q)$ is a circulant AST if $q$ is even, and  the AST in \cite[Proposition 4.7]{MB90} constructed from the $1$-dimensional affine group $\AGL(1,q)$ is  a circulant AST if $q$ is prime.
 
In Definition~\ref{d:AST-reg} we introduce a special kind of partition of the set 
\begin{equation}\label{d:X}
X:=(\Omega\setminus\{0\})^{[2]}=\{ (i,j)\mid i, j\in\Omega, i\ne 0, j\ne 0, i\ne j\}.
\end{equation}
called an \emph{AST-regular partition}, and we show (Theorem~\ref{p:ASTcirc}) how to construct a circulant AST $\mathcal{A}(\mathcal{I})$ based on $\Omega$ from a given AST-regular partition $\mathcal{I}$ of $X$. It turns out that this construction is quite general, and we show that each circulant AST based on $\Omega$ arises in this way.

\begin{theorem}\label{t:circAST}
Let $\Omega=\{0,\dots,n-1\}$ and let $X$ be as in \eqref{d:X}. Then $\mathcal{A}$ is a circulant AST based on $\Omega$ if and only if  $\mathcal{A}= \mathcal{A}(\mathcal{I})$ for some AST-regular partition $\mathcal{I}$ of $X$.
\end{theorem} 

Each of the nontrivial $3$-circulants $R_I$ in a circulant AST $\mathcal{A}(\mathcal{I})$ has size $n(n-1)n_I$ where $n_I$ is the parameter in Definition~\ref{d:AST-reg}(a). It is tantalising to wonder if there are many circulant ASTs for which all these parameters $n_I$ are equal to $1$. While we have not made much progress in resolving this, we have introduced the concept of a \emph{thin $3$-circulant} in Definition~\ref{d:thin}. In particular each  nontrivial $3$-circulant $R_I$ in a circulant AST for which $n_I=1$ is thin. Moreover, using a graph theoretic model for nontrivial $3$-circulants, and applying a famous result on perfect matchings in graphs, we have shown (Theorem~\ref{prop-matching}) quite generally that every nontrivial $3$-circulant $R_I$ in a circulant AST is a disjoint union of $n_I$ nontrivial thin $3$-circulants.  

Theorem~\ref{t:circAST} is proved in Section~\ref{s:circAST}, and we explore thin $3$-circulants in Section~\ref{s:thin}. In our final section~\ref{sec:conclusion} we consider several open questions suggested by this work. In particular,  
association schemes on triples (ASTs) are intimately related to the ternary algebras introduced by the first author and Mesner in \cite{MB90} and developed further in \cite{MB94}. We plan to develop the ideas presented in this paper further by studying the ternary algebras associated with circulant ASTs.

\section{AST-regular partitions}\label{s:circAST}

The definition of an AST-regular partition of the set $X$ in \eqref{d:X} involves the following maps:
\begin{itemize}
\item the natural projection maps $\pi_1, \pi_2:X\to \Omega$ given by $\pi_1:(i,j)\to i,\ \pi_2:(i,j)\to j$;
\item the transpose map $T$, and another map $\tau$ on subsets $I\subseteq X$ given by 
\[
I^T:=\{(j,i)\mid (i,j)\in I\},\quad \mbox{and}\quad I^\tau := \{(-i, j-i) \mid (i,j)\in I \}.
\]
\end{itemize} 

We note some basic properties, namely a link between  circulant $3$-relations and subsets of $X$, and an action of $\Sym(3)$ on subsets of $X$. For isomorphic groups $G$ and $H$, acting on sets $Y$ and $Z$, respectively, a \emph{permutational isomorphism} is a pair $(\rho, f)$ such that $\rho:G\to H$ is a group isomorphism, $f:Y\to Z$ is a bijection, and for all $g\in G$ and $y\in Y$, $(y^g)f = (yf)^{(g)\rho}$. 

\begin{lemma}\label{l:circ1}
Let $X$ be as in \eqref{d:X}, where $\Omega=\{0,\dots,n-1\}$, let  $\mathcal{R}$ be the set of nontrivial $3$-circulants on $\Omega$ and $\mathcal{X}$ be the set of non-empty subsets of $X$.
\begin{enumerate}
\item[(a)] A ternary relation $R$ on $\Omega$ is a  nontrivial $3$-circulant if and only if there exists a (unique) non-empty subset $I\subseteq X$ such that $R=R_I$, where
\begin{equation}\label{e:RI}
R_I=\{(x,i+x,j+x)\mid x\in\Omega, (i,j)\in I\}. 
\end{equation}

\item[(b)] The map $f:R_I\to I$ is a bijection $\mathcal{R}\to \mathcal{X}$, and the map  $\rho: (12)\to \tau,\ (23)\to T$ extends uniquely to a group isomorphism $\rho$ from $\Sym(3)=\langle (12), (23)\rangle$ to $\langle \tau, T\rangle$.

\item[(c)] The pair $(\rho, f)$ is a permutational isomorphism from the $\Sym(3)$-action  on $\mathcal{R}$ to the $\langle \tau, T\rangle$-action on $\mathcal{X}$. In particular, the images of $I\in\mathcal{X}$ under the non-identity elements of $\Sym(3)$ are as follows (note that $\tau T\tau = T\tau T$):
\[
\begin{array}{lllll}
I^{(12)\rho} &= &I^\tau	&= &\{(-i, j-i) \mid (i,j)\in I \}  \\
I^{(23)\rho} &=&I^T &= &\{(j,i) \mid (i,j)\in I \}  \\
I^{(13)\rho} &= &I^{T\tau T}	&=&\{(i-j, -j) \mid (i,j)\in I \} \\
I^{(123)\rho} &= &I^{T\tau}	&=&\{(-j, i-j) \mid (i,j)\in I \}  \\
I^{(132)\rho} &= &I^{\tau T}	&= &\{(j-i, -i) \mid (i,j)\in I \}.
 \end{array}
\]
\end{enumerate}
 \end{lemma}

\begin{proof}
(a) Let $R\in\mathcal{R}$. By definition of a $3$-circulant $R$ contains a triple of the form $(0,i,j)$, and since $R$ is nontrivial, for each such triple we have $(i,j)\in X$. Let $I$ be the (non-empty) set of all pairs $(i,j)$ such that $(0,i,j)\in R$. Then, by \eqref{d:3circ}, $R$ contains the set $\{(x,i+x,j+x)\mid x\in\Omega, (i,j)\in I\}$, and equality holds by the definition of $I$.
Conversely, if follows from the definition of $R_I$ that for each non-empty $I\subseteq X$, $R_I$ is a nontrivial $3$-circulant.

(b) It follows from part (a) that $f$ is a bijection, and it is straightforward to check, from the definitions of $T$ and $\tau$ given above, that $T^2=\tau^2=(T\tau)^3=1$ so that  $\rho:(12)\to \tau,\ (23)\to T$ determines a group isomorphism $\Sym(3)\to \langle T, \tau\rangle$. In particular $\tau T\tau = T\tau T$. 

(c) It is straightforward, if tedious, to check that, for all $g\in\Sym(3)$ and $R_I\in \mathcal{R}$, the image $R_I^g$ is equal to $R_J$ where $J=I^{(g)\rho}$, and these subsets $I^{(g)\rho}$ are as listed in the statement for $g\ne 1$.  Thus 
$(R_I^g)f = (R_If)^{(g)\rho}$, so $(\rho, f)$ is a permutational isomorphism. 
\end{proof}

Now we define an AST-regular partition, and show that each AST-regular partition gives rise to a circulant AST.

\begin{definition}\label{d:AST-reg}
{\rm 
A partition $\mathcal{I}$ of the set $X$ in \eqref{d:X} is said to be \emph{AST-regular} if the following three conditions hold.
\begin{enumerate}
\item[(a)] For each $I\in\mathcal{I}$, $\pi_1(I)=\pi_2(I)=\Omega\setminus\{0\}$, and there exists a positive integer $n_I$ such that, for each $x\in \Omega\setminus\{0\}$, there are exactly $n_I$ elements $j\in\Omega\setminus\{0,x\}$ with $(x,j)\in I$ and   exactly $n_I$ elements $i\in\Omega\setminus\{0,x\}$ with $(i,x)\in I$.  In particular  $|I|=(n-1)n_I$ and $|R_I|=n(n-1)n_I$ (with $R_I$ as in \eqref{e:RI}).

\item[(b)] The partition $\mathcal{I}$ is invariant under the action of $\Sym(3)$ given in Lemma~\ref{l:circ1}. 

\item[(c)]  For all (not necessarily distinct) $I,J, K, L\in\mathcal{I}$, there exists a non-negative integer $p^L_{IJK}$ such that, for all $(y,z)\in L$, there are exactly $p^L_{IJK}$ elements $w\in\Omega\setminus\{0,y,z\}$ such that 
\[
(y-w, z-w)\in I,\quad (w,z)\in J,\quad \mbox{and}\quad (y,w)\in K.
\]
\end{enumerate}  
}
\end{definition}

\begin{theorem}\label{p:ASTcirc}
Let $\Omega=\{0,\dots,n-1\}$, let $X$ be as in \eqref{d:X}, and let $\mathcal{I}$ be an AST-regular partition of $X$. Then $\mathcal{A}(\mathcal{I}):=\{R_0, R_1, R_2, R_3\}\cup\{ R_I\mid I\in\mathcal{I}\}$ is a circulant AST, with the relations $R_I$ as in \eqref{e:RI}.
\end{theorem}

\begin{proof}
Let $\mathcal{A}= \mathcal{A}(\mathcal{I})$. As we noted after \eqref{eq:3circ}, each $R_i$ for $0\leq i\leq 3$, is a $3$-circulant, and by Lemma~\ref{l:circ1}(a), each $R_I\in\mathcal{A}$ is also a $3$-circulant. It remains to prove that each of the conditions A1, A2, A3 of Definition~\ref{d:3circ} is valid. Let $I\in\mathcal{I}$. By Definition~\ref{d:AST-reg}(a), there is a positive integer $n_I$ such that, for each non-zero $x\in\Omega$, there are exactly $n_I$ elements $(x,j)\in I$, or equivalently, exactly $n_I$ triples $(0,x,j)\in R_I$. Thus for each $x,y\in \Omega$ with $x\ne y$, there are exactly $n_I$ elements $j\in \Omega\setminus\{x, y\}$ such that $(0,y-x,j-x)\in R_I$ and hence, by Lemma~\ref{l:circ1}(a), exactly $n_I$ elements $j$ such that $(x,y,j)\in R_I$. Therefore condition A1 holds. By Definition~\ref{d:AST-reg}(b), $\Sym(3)$ induces an action on the parts of $\mathcal{I}$ and hence on the nontrivial relations in $\mathcal{A}$, and since $\Sym(3)$ also acts naturally on $\{R_1, R_2, R_3\}$ and fixes $R_0$, it follows that condition A3 holds. 

It remains to confirm condition A2. We must consider all $I, J, K, L$, not necessarily distinct, in $\mathcal{I}\cup\{0,1,2,3\}$. 
If all of $I, J, K, L$ lie in $\mathcal{I}$ then the condition in A2 follows from Definition~\ref{d:AST-reg}(c). 
Also, if $\{I,J,K,L\}\subseteq \{0,1,2,3\}$, then the condition in A2 holds as these trivial relations are the same in every AST. So we may assume that at least one, but not all of $I, J, K, L$ lies in
$\mathcal{I}$. 

If  $0\in\{I,J,K,L\}$ then, we claim that there are no elements $w$ satisfying the inclusions in A2 for any triple in $R_L$, and hence A2 holds for  these values of $I, J, K, L$. If $L=0$, then there are no elements $w$  for which the A2 inclusions hold for any $(x,x,x)\in R_L$, since one of the other relations is nontrivial and hence consists of triples with parirwise distinct entries. On the other hand if, without loss of generality, $I=0$, and if the inclusions in A2 hold for some triple $(x,y,z)\in R_L$ and some element $w$, then $w=y=z$, and hence all of  $I,J,K,L\in \{0,1,2,3\}$, which is a contradiction. This proves the claim. 

The remaining cases to be are considered are those in which $0\not\in \{I,J,K,L\}$, and at least one of $I,J, K, L$ lies in $\{1,2,3\}$ and at least one lies in $\mathcal{I}$. 

Suppose first that $L\in\{1,2,3\}$. Using condition A3 we may assume that $L=1$. Assume also that, for some triple $(x,y,y)\in R_L$, there exists $w$ such that the inclusions in A2 hold. In particular $(w,y,y)\in R_I$, so $I=1$ and $w\ne y$ (as $I\ne 0$).  Also $(x,w,y)\in R_J$ and $(x,y,w)\in R_K$, and since at least one of $R_J, R_K$ is nontrivial, it follows that $x,y,w$ are pairwise distinct, and hence both $J, K$ lie in $\mathcal{I}$. 
Now $\mathcal{I}$ is a partition of $X$ and contains $J^T$, by Definition~\ref{d:AST-reg}(b). Since $(x,w,y)\in R_J$ we also have $(0,w-x, y-x)\in R_J$ (since $R_J$ is a $3$-circulant), so $(w-x,y-x)\in J$ and hence $(y-x,w-x)\in J^T$ which implies that $(x,y,w)\in R_{J^T}$. However $(x,y,w)\in R_K$, and therefore $K=J^T$. It now follows using Definition~\ref{d:AST-reg}(a) that, for all $(x,y,y)\in R_L$, the number of $w$ for which the A2 inclusions hold is equal to $n_K=n_J$, so condition A2 holds in this case. 

Thus we may assume that $L\in\mathcal{I}$ so $R_L$ is nontrivial. Assume that, for some triple $(x,y,z)\in R_L$ (so $x,y,z$ are pairwise distinct), there exists $w$ such that the inclusions in A2 hold. Since at least one of $R_I, R_J, R_K$ is trivial we must have $w\in\{x,y,z\}$, and there is only one possible value of $w$. Consider the case where $w=x$. Here $(x,y,z)\in R_I$ so $R_I=R_L$ is nontrivial, $(x,x,z)\in R_J$ so $R_J=R_3$, and $(x,y,x)\in R_K$ so $R_K=R_2$. Then for each $(x',y',z')\in R_L=R_I$, there is a unique element $w'$ for which the A2 inclusions hold for this triple, namely $w'=x'$. Thus condition A2 holds in this case also. The arguments proving condition A2 for the cases $w=y$ and $w=z$ are similar. This completes the proof that $\mathcal{A}$ is an AST.  
\end{proof}

Finally in this section we prove Theorem~\ref{t:circAST}.

\begin{proof}(Proof of Theorem~\ref{t:circAST})\quad 
Let $\Omega=\{0,\dots,n-1\}$ and let  $X$ be as in \eqref{d:X}. 
By Theorem~\ref{p:ASTcirc}, $\mathcal{A}(\mathcal{I})$ is a circulant AST for each AST-regular partition $\mathcal{A}$ of $X$. So to prove Theorem~\ref{t:circAST} we assume that $\mathcal{A}$ is a circulant AST based on $\Omega$ and show that $\mathcal{A}$  is equal to $\mathcal{A}(\mathcal{I})$ for some AST-regular partition $\mathcal{I}$ of $X$. 
By Lemma~\ref{l:circ1}(a), each nontrivial relation in $\mathcal{A}$ is equal to $R_I$ for some non-empty subset $I\subseteq X$. Since the nontrivial relations form a partition of $\Omega^{[3]}=\{(x,y,z)\mid x,y,z \ \mbox{pairwise distinct}\}$, it follows that the subsets $I$ form a partition $\mathcal{I}$ of $X$. Thus $\mathcal{A} =\{R_0, R_1, R_2, R_3\}\cup\{ R_I\mid I\in\mathcal{I}\}$ and we need to check that $\mathcal{I}$ satisfies conditions (a)--(c) of Definition~\ref{d:AST-reg}.

Consider $I\in\mathcal{I}$. It follows from condition A1 of Definition~\ref{d:3circ} that there is a positive constant $n_I$ such that, for each $x\ne 0$, there are exactly $n_I$ elements $j\in \Omega\setminus\{ 0, x\}$ with $(0,x,j)\in R_I$, or equivalently, $(x,j)\in I$. (We note that $n_I$ is the parameter $n_i^{(3)}$ of A1, where $R_I$ is the relation  $R_i$ of A1.) In particular,  $\pi_1(I)=\Omega\setminus\{0\}$ and $|R_I|=n(n-1)n_I$. By condition A3 of Definition~\ref{d:3circ}, and taking $\sigma=(23)$, $\mathcal{A}$ contains the relation
\[
R_{I}^\sigma = \{(x_1, x_3, x_2) \mid (x_1, x_2, x_3)\in R_I\},
\] 
and from the description of $R_I$ in Lemma~\ref{l:circ1}(b), it follows that $R_{I}^\sigma = R_{I^T}$ where $I^T=\{(j,i)\mid (i,j)\in I\}$ is the image of $I$ under the transpose map. This proves that $I^T\in\mathcal{I}$. Moreover, $\pi_2(I)=\pi_1(I^T)$, and we have just proved that the latter must be equal to $\Omega\setminus\{0\}$ since $I^T\in\mathcal{I}$.  Condition A1 applied to $R_{I^T}$ yields that the number $s$ of triples $(0,j,x)\in R_{I^T}$ for a fixed $j$, is independent of the choice of $j\ne 0$. Thus $|R_{I^T}|=n(n-1)s$, and since $|R_{I^T}|=|R_I|$ it follows that $s=n_I$ (which therefore also equals $n_{I^T}$), proving condition (a) of Definition~\ref{d:AST-reg}.

To prove Definition~\ref{d:AST-reg}(b) we use condition A3 of Definition~\ref{d:3circ} and apply each element of $\Sym(3)$ to $R_I$ using Lemma~\ref{l:circ1}(c). For example, taking $\sigma=(12)$, the relation $R_I^\sigma$ lies in $\mathcal{A}$ by condition A3, and consists of all the triples $(x+i, x, x+j)$ with $x\in\Omega$ and $(i,j)\in I$; this is precisely the set of triples in the relation $R_{J}$ with $J=I^{(12)\rho}$ as given in Lemma~\ref{l:circ1}(c). The fact that $\mathcal{I}$ is closed under the action of the other elements of $\Sym(3)$ follows by  analogous arguments.

Consider (not necessarily distinct) $I,J, K, L\in\mathcal{I}$. We use condition A2 of Definition~\ref{d:3circ} on the relations $R_I, R_J, R_K, R_L$ to prove Definition~\ref{d:AST-reg}(c). First we note that, if A2 holds with the constant $p=p^L_{IJK}$ for all triples $(0,y,z)\in R_L$, then it holds with the same constant $p$ for all triples $(x,y,z)\in R_L$. So let $(0,y,z)\in R_L$, or equivalently $(y,z)\in L$. (Such triples exist since $L\in\mathcal{I}$.) Then by A2, there exist $p=p^L_{IJK}$ (possibly zero) elements $w\in\Omega$ such that 
$(w,y,z)\in R_I, (0,w,z)\in R_J$ and $(0,y,w)\in R_K$. We note that these three inclusions are equivalent to  $(y-w, z-w)\in I$, $(w,z)\in J$, and $(y,w)\in K$, respectively. Also, since $J, K\subseteq X$, each such element $w$ must be distinct from $0, y, z$. This proves Definition~\ref{d:AST-reg}(c), and completes the proof of Theorem~\ref{t:circAST}.
\end{proof}

\section{Thin $3$-circulants}\label{s:thin}

As discussed in the introduction, a nontrivial $3$-circulant $R_I$, lying in a circulant AST based on a set $\Omega$ of size $n$, contains exactly $n(n-1)n_I$ triples, for some positive integer $n_I$. If $n_I=1$ then $R_I$ is a thin $3$-circulant, which we define as follows. 

\begin{definition}\label{d:thin}
{\rm 
Let $ab\in\{12, 13, 23\}$. Then a $3$-circulant $R$ based on a set $\Omega$  is said to be \emph{$ab$-thin} if the map
$\sigma_{ab}:R\to \Omega^2$ given by $\sigma_{ab}:(x_1,x_2,x_3)\to (x_a,x_b)$ is one-to-one and has image 
$\Omega^{[2]}=\{(x,y)\mid x,y\in\Omega, x\ne y\}$. We say that $R$ is \emph{thin} if it is $ab$-thin for some $ab$.
}
\end{definition}

Note that, by definition, each thin $3$-circulant contains exactly $|\Omega^{[2]}|=n(n-1)$ triples. First we study the structure of an arbitrary thin $3$-circulant, and in particular we see that each of the trivial relations $R_1, R_2, R_3$ in an AST is thin. 
A \emph{derangement} of a set $\Omega$ is a permutation of $\Omega$ which has no fixed points.  

\begin{lemma}\label{l:thin1}
Let $\Omega=\{0,1,\dots, n-1\}$ and let $X$ be as in \eqref{d:X}.
\begin{enumerate}
\item[(a)] Each of $R_1, R_2, R_3$ defined in Definition~\ref{AST1} is a thin $3$-circulant, and moreover each of them is $ab$-thin for exactly two pairs $ab\in\{12, 13, 23\}$.

\item[(b)] Let $ab\in\{12, 13, 23\}$ and suppose that $\{1,2,3\}=\{a,b,c\}$. Then a $3$-circulant $R$ based on $\Omega$ is $ab$-thin if and only if there exists a map $\rho:\Omega\setminus\{0\}\to\Omega$ such that $R$ consists precisely of the $n(n-1)$ triples $\mathbf{x}$ which satisfy
\[
\mathbf{x}_d=\left\{\begin{array}{ll}
	x 	& \mbox{if}\ d=a \\
	x+y & \mbox{if}\ d=b \\
	x+y^\rho & \mbox{if}\ d=c, \\
\end{array}\right.
\]
for some $(x, x+y)\in\Omega^{[2]}$. Moreover such an $ab$-thin $3$-circulant $R$ is nontrivial if and only if $\rho$ is a derangement of $\Omega\setminus\{0\}$.  


\item[(c)] If $I\subseteq X$, and if both $I$ and its image $I^\tau$ satisfy the conditions of Definition~\ref{d:AST-reg}(a), with parameter $n_I=1$ (with $\tau$ as in Lemma~\ref{l:circ1}(c)), then the $3$-circulant $R_I$ is $ab$-thin for all three pairs $ab\in\{12,13, 23\}$.
\end{enumerate}
\end{lemma}

\begin{proof}
(a)  For $R_1=\{(x,y,y)\mid x,y\in\Omega, x\ne y\}$, it is straightforward to check that $\sigma_{12}$ and $\sigma_{13}$ satisfy the requirements of Definition~\ref{d:thin}, but of course $\sigma_{23}$ does not. In particular $R_1$ is thin. Proofs for the $3$-circulants $R_2, R_3$ are similar. 

(b) Recall from Definition~\ref{d:thin} that $(x, x+y)\in\Omega^{[2]}$ if and only if $y\in \Omega\setminus\{0\}$. Thus if, for some $ab\in\{12, 13, 23\}$, a $3$-circulant $R$ consists of the triples in part (b) for some map $\rho$ then $\sigma_{ab}: R\to \Omega^{[2]}$ is onto, and also  $|R|=n(n-1)$; and it follows that $\sigma_{ab}$ is one-to-one. Hence $R$ is $ab$-thin. Suppose conversely that $R$ is an $ab$-thin $3$-circulant, and for each $x\in\Omega$ let $R(x):=\{\mathbf{x}\in R \mid \mathbf{x}_a=x\}$. Since $R$ is a $3$-circulant, the map $\mathbf{x}\to \mathbf{x}+(x,x,x)$ defines a bijection from $R(0)$ to $R(x)$, and hence $|R(x)|=|R|/n=n-1$ for each $x$. 
By Definition~\ref{d:thin}, $R(0)$ comprises a unique triple $\mathbf{x}$ with $\mathbf{x}_a=0, \mathbf{x}_b=y$ for each $y\ne 0$. Define $y^\rho\in\Omega$ to be the entry $\mathbf{x}_c$ in this triple. Then $R$ is as in part (b) for this map $\rho$.  

Suppose now that $R$ is an $ab$-thin $3$-circulant relative to the map $\rho$ above. Then $R$ is nontrivial if and only if $\mathbf{x}_c\ne \mathbf{x}_a, \mathbf{x}_b$ for all $\mathbf{x}\in R$, or equivalently, $y^\rho\ne 0, y$ for all $y\in\Omega\setminus\{0\}$, that is to say, $\rho$ is a derangement of $\Omega\setminus\{0\}$.

(c) Suppose that $I\subseteq X$ satisfies the conditions of Definition~\ref{d:AST-reg}(a) with parameter $n_I=1$. Then it follows from \eqref{e:RI} and Definition~\ref{d:AST-reg}(a)  that $R_I=\{(x, i+x, j+x)\mid x\in\Omega, (i,j)\in I\}$ projects onto $\Omega^{[2]}$ under both $\sigma_{12}$ and $\sigma_{13}$. Also these maps are one-to-one since $|R_I|=n(n-1)n_I=n(n-1)$. Thus $R_I$ is $12$-thin and $13$-thin. Now assume also that 
$J:=I^\tau$ satisfies the conditions of Definition~\ref{d:AST-reg}(a). It follows from Lemma~\ref{l:circ1} that applying $\tau$ to $I$ corresponds to applying the permutation $(12)\in\Sym(3)$ to $R_I$, and so 
$R_J = R_I^{(12)}  = \{ (x_2, x_1, x_3) \mid (x_1,x_2,x_3)\in R_I\}$. In particular $|R_J|=|R_I|$, so $n_J=n_I=1$. Our argument above applied to $J$ instead of $I$ shows that $R_J$ is $12$-thin and $13$-thin. Now the map  $\sigma_{13}$ for $R_J$ maps each triple $(x_2, x_1, x_3)\in J$ to $(x_2,x_3)$, and as  $R_J$ is $13$-thin, this map is one-to-one and has image $\Omega^{[2]}$. It follows that the map $\sigma_{23}: (x_1,x_2,x_3)\to (x_2,x_3)$ for $R_I$ is  one-to-one and has image $\Omega^{[2]}$, and hence $R_I$ is $23$-thin. 
\end{proof}

Finally in this section we prove that each nontrivial $3$-relation in a circulant AST is a disjoint union of thin $3$-circulants. This result follows from a famous graph theoretic result, due to K\"{o}nig~\cite{konig}, on perfect matchings (which we explain carefully, in context, in the proof). 

\begin{theorem}\label{prop-matching}
Let $\Omega=\{0,1,\dots, n-1\}$ and let $X$ be as in \eqref{d:X}. 
\begin{enumerate}
\item[(a)] Let $I\subseteq X$ such that the conditions of Definition~\ref{d:AST-reg}(a) hold. Then $R_I$ is a disjoint union of $n_I$ nontrivial $3$-circulants each of which is $12$-thin and $13$-thin.

\item[(b)] If $\mathcal{A}$ is a circulant AST based on $\Omega$, then part (a) holds for each nontrivial $3$-circulant in $\mathcal{A}$.
\end{enumerate}
\end{theorem}

\begin{proof}
(a) 
Recall that $I$ is a subset of the set $X$ defined in \eqref{d:X} such that (i) for each $x\ne 0$ there are exactly $n_I$ elements $y$ such that $(x,y)\in I$, (ii) for each $y\ne 0$ there are exactly $n_I$ elements $x$ such that $(x,y)\in I$, and (iii) $\pi_1(I)$ and $\pi_2(I)$ are both equal to $\Omega\setminus\{0\}$. 

Such a subset $I$ may be interpreted as the set of edges of a subgraph $\mathcal{G}_I$ of the complete bipartite graph $\Gamma=K_{n-1,n-1}$ with vertex set $\mathbb{Z}_2\times (\Omega\setminus\{0\})$ and edges of the form $\{(0,x), (1,y)\}$ for all $x,y\in\Omega\setminus\{0\}$. The subgraph $\mathcal{G}_I$ is obtained by identifying each pair $(x,y)\in I$ with the edge joining $(0,x)$ and $(1,y)$. Conditions (i)-(iii) ensure that each vertex of $\Gamma$ lies on exactly $n_I$ edges of $\mathcal{G}_I$,  that is to say, the subgraph $\mathcal{G}_I$ is bipartite and regular of valency $n_I$.

A perfect matching in $\mathcal{G}_I$ is defined as a subset $J$ of edges of $\mathcal{G}_I$ such that each vertex of $\mathcal{G}_I$ lies on exactly one edge in $J$. In other words, a perfect matching $J$ of $\mathcal{G}_I$ under our identification of edges corresponds to a subset of $I$ (which abusing notation slightly we will also call $J$) for which the corresponding  $3$-circulant $R_J=\{(x, i+x, j+x)\mid x\in\Omega, (i,j)\in J\}$ is of the form given in Lemma~\ref{l:thin1}(b) with $ab=12$ and the  map $\rho$ given by $i^\rho=j$ if and only if $(i,j)\in J$. Note that $\rho$ is well-defined since $J$ projects onto $\Omega\setminus\{0\}$ in its first coordinate and the map $\sigma_{12}$ for $J$ is one-to-one. Also $|J|=n-1$ as $J$ is a perfect matching, and  the image of $\rho$ is $\Omega\setminus\{0\}$, again since $J$ is a perfect matching in $\mathcal{G}_I$. We conclude that $R_J$ is $12$-thin. Similarly  $R_J$ is $13$-thin since $R_J$ is of the form given in Lemma~\ref{l:thin1}(b) with $ab=13$ relative to the  map $\rho^{-1}$.  Hence a perfect matching $J$ in $\mathcal{G}_I$ corresponds to a subset $R_{J}$ of $R_I$ which, by Lemma~\ref{l:thin1}(b), is a  $12$-thin and $13$-thin $3$-circulant. Also $R_J$ is nontrivial since $R_J\subseteq R_I$.
    
For the final part of the proof, we apply K\"{o}nig's theorem~\cite{konig} that the edge set of a finite regular bipartite graph $\mathcal{G}_I$ of valency $n_I$ may be partitioned as a disjoint union of $n_I$ perfect matchings $J_1, \dots, J_{n_I}$. Recent proofs of this result can be found in  \cite[Theorem 3]{NP} and \cite[Theorem 3.5.1]{GR}, see also the discussion in \cite[Remark 1.11(a)]{IP}.  
We conclude that $R_I$ is the disjoint union of  $3$-circulants $R(J_i)$, for $i=1,\dots, n_I$, each of which is nontrivial, $12$-thin and $13$-thin.

(b) By Theorem~\ref{t:circAST}, $\mathcal{A}=\mathcal{A}(\mathcal{I})$ for some AST-regular partition $\mathcal{I}$ of $X$. Thus each nontrivial $3$-circulant in $\mathcal{A}$ satisfies the conditions of Definition~\ref{d:AST-reg}(a), and the result now follows from part (a). 
\end{proof}

\section{Conclusions and new problems}\label{sec:conclusion}

Our study of circulant ASTs has raised a number of questions we believe are worthy of further study, and would shed more light on these interesting structures. Our main result Theorem~\ref{t:circAST} gives a complete characterisation of circulant ASTs in terms of AST-regular partitions of the set $X$ in \eqref{d:X}. This is a significant `reduction result' which should make it easier to find examples. So our first problem is to do just this, since most of the known examples arise as orbits on triples of $2$-transitive permutation groups. It would be an excellent outcome to find new examples.

\begin{problem}
Find new families of circulant ASTs.
\end{problem}

The automorphism group of a circulant AST contains the cyclic subgroup 
$A=\langle a\rangle$ used in Subsection~\ref{sub:circ} in the definition of $3$-circulants, and $A$ is  transitive on the base set $\Omega$. The paper \cite{Li2} gives a general structural picture of permutation groups containing transitive cyclic subgroups, extending the classification due to W. Feit and G. A. Jones of the primitive groups with this property. Although all the primitive examples are either affine of prime degree or $2$-transitive there are some other kinds of groups identified in \cite{Li2} which might give a starting place for new constructions.
On the other hand AST-regular partitions have a new data structure which has not yet been studied systematically, even for the known ASTs.

\begin{problem}
Describe the AST-regular partitions of $X$ corresponding to the known families of circulant ASTs arising from $2$-transitive permutation groups.
\end{problem}
 
Our second main result Theorem~\ref{prop-matching} hints at the importance of nontrivial thin $3$-circulants in the structure of circulant ASTs. The graph theoretic result we used to prove this, namely that the edge set of a regular bipartite graph could be decomposed as a disjoint union of perfect matchings, does not give a canonical way of doing this. Our first question after proving this result was whether there were examples of circulant ASTs in which each nontrivial $3$-circulant is thin? The unique AST based on a set of size $3$ clearly has this property, but can this happen for larger values of $n$?

\begin{problem}
Find examples of circulant ASTs in which each nontrivial $3$-circulant is thin.
\end{problem}

It is possible that this condition is undesirably stringent, and should be relaxed.
A lot of attention has been given to constructing \emph{symmetric ASTs} $\mathcal{A}$, namely those in which each nontrivial ternary relation in $\mathcal{A}$ is invariant under the action of $\Sym(3)$. The structure of symmetric ASTs was studied in \cite{MB90}, and from a different point of view by Klin, Mesner and Woldar~\cite{Klin2} where the focus was on combinatorial versions of a transitive extension of a generously transitive permutation group. (A transitive group is \emph{generously transitive} if each of its orbits on ordered pairs is invariant under the transpose map $T$ defined in Section~\ref{s:circAST}, see~\cite{PMN75}.) 
If we were seeking symmetric circulant ASTs, then a natural special case would be those in which each nontrivial $3$-circulant is the symmetrisation of a thin $3$-circulant (the smallest $\Sym(3)$-invariant $3$-circulant containing the thin $3$-circulant).  

\begin{problem}
Find examples of symmetric circulant ASTs in which each nontrivial $3$-circulant is the symmetrisation of a thin $3$-circulant.
\end{problem} 


As suggested in \cite{MB90} and also in \cite{Klin2}, it would be possible to generalise the results for ASTs to higher dimensional schemes, that is those consisting of $k$-ary relations (subsets of $\Omega^k$) for some $k>3$ (see, for example, \cite{Cooper, Gnang17, Gnang20, Lincoln}).  Finally, as we mentioned in the introduction, we hope to explore in future work the ternary algebras corresponding to circulant ASTs, refining the work of Mesner and the second author in \cite{MB90, MB94}.

\end{document}